\documentclass[runningheads,envcountsame,a4paper]{llncs} 

\usepackage{amsthm,amssymb,amsmath}
\usepackage{wasysym}
\usepackage{graphicx}
\usepackage{latexsym}
\usepackage[english]{babel}
\usepackage[utf8]{inputenc}
\usepackage{csquotes}
\usepackage{algorithm} 
\usepackage{algpseudocode}
\usepackage{tikz}
\usetikzlibrary{arrows,shapes.geometric,positioning,calc}
\usepackage{changepage}
\usepackage{subfig}
\usepackage{soul}
\usepackage{todonotes}
\usepackage{enumerate}
\usepackage{authblk}
\usepackage{cite}
\makeatletter
\renewcommand*\l@author[2]{}
\renewcommand*\l@title[2]{}
\makeatletter

\usepackage{stackengine}

\RequirePackage[colorlinks=true,
    pdfstartview=FitV,
    linkcolor=blue,
    citecolor=blue,
    urlcolor=blue,
    breaklinks=true,
    pdfencoding=auto,
    hyperfootnotes=false]{hyperref}

\RequirePackage{aliascnt}
\RequirePackage{cleveref}

\newtheorem{thm}[theorem]{Theorem}
    \Crefformat{thm}{\bfseries theorem~#2#1#3}
    \Crefformat{thm}{\bfseries Theorem~#2#1#3}
\newtheorem{cor}[theorem]{Corollary}
    \Crefformat{cor}{\bfseries corollary~#2#1#3}
    \Crefformat{cor}{\bfseries Corollary~#2#1#3}
\newtheorem{lem}[theorem]{Lemma}
    \Crefformat{lem}{\bfseries lemma~#2#1#3}
    \Crefformat{lem}{\bfseries Lemma~#2#1#3}
\newtheorem{prop}[theorem]{Proposition}
    \Crefformat{prop}{\bfseries proposition~#2#1#3}
    \Crefformat{prop}{\bfseries Proposition~#2#1#3}
    
    \Crefformat{conj}{\bfseries conjecture~#2#1#3}
    \Crefformat{conj}{\bfseries Conjecture~#2#1#3}
    
\theoremstyle{definition}
\newaliascnt{defn}{theorem}

    \Crefformat{defn}{\bfseries definition~#2#1#3}
    \Crefformat{defn}{\bfseries Definition~#2#1#3}
\newtheorem{exa}[theorem]{Example}
    \Crefformat{exa}{\bfseries example~#2#1#3}
    \Crefformat{exa}{\bfseries Example~#2#1#3}
\newtheorem{rmk}[theorem]{Remark}

\newsavebox{\qedB}

\sbox{\qedB}{\setlength{\unitlength}{1mm}
 \begin{picture}(4,4)(0,0)
  \thinlines
  {\put(0,0){\framebox(2.83,2.83){}}}%
  {\put(1.17,1.17){\framebox(2.83,2.83){}}}%
  {\put(0,0){\framebox(4,4){}}}%
  {\put(1.17,1.17){{\rule{1ex}{1ex} }}}%
 \end{picture}}
 
\newcommand{\QEDB}{\ifmmode\def\next{\tag"\usebox{\qedB}"}%
 \else\let\next=\relax
 {\unskip\nobreak\hfil\penalty50
 \hskip2em\hbox{}\nobreak\hfil\usebox{\qedB}
 \parfillskip=0pt \finalhyphendemerits=0\penalty-100\bigskip}\fi\next}

\newcommand{\N}{\mathbb N}

\newcommand{\C}{\mathbb{C}}
\newcommand{\bprop}{\begin{prop}}
\newcommand{\eprop}{\end{prop}}
\newcommand{\bcor}{\begin{cor}}
\newcommand{\ecor}{\end{cor}}
\newcommand{\blem}{\begin{lem}}
\newcommand{\elem}{\end{lem}}

\newcommand{\ceil}[1]{\left\lceil#1\right\rceil}
\newcommand{\floor}[1]{\left\lfloor#1\right\rfloor}
\newcommand{\h}[1]{\hspace{#1pt}}
\newcommand{\ol}[1]{\overline{#1}}

\renewcommand{\floor}[1]{\lfloor #1 \rfloor}

\title{On the Partial Sum of Subword-Counting Sequences}

\author{Pranjal Jain\inst{1} \and Shuo Li\inst{2}}

\institute{Department of Mathematics,\\
Indian Institute of Science Education and Research,\\
Pune, India\\
\email{pranjal.jain@students.iiserpune.ac.in} 
\and
Department of Mathematics and Statistics,\\
The University of Winnipeg,\\
Winnipeg (MB), Canada\\
\email{shuo.li.ismin@gmail.com}
}

\begin{document}
    
\maketitle

\begin{abstract}
Let $w$ be a finite word over the alphabet $\{0,1\}$. For any natural number $n$, let $s_w(n)$ denote the number of occurrence of $w$ in the binary expansion of $n$ as a scattered subsequence. We study the behavior of the partial sum $\sum_{n=0}^N(-1)^{s_w(n)}$ and characterize several classes of words $w$ satisfying $\sum_{n=0}^N(-1)^{s_w(n)}= O(N^{1-\epsilon})$ for some $\epsilon >0$.
\end{abstract}
\begingroup
\let\clearpage\relax
\tableofcontents*
\endgroup
\newpage
\section{Introduction}

Let $w$ be a finite word over the alphabet $\{0,1\}$. For any natural number $n$, let $s_w(n)$ and $e_w(n)$ denote respectively the number of occurrence of $w$ in the binary expansion of $n$ as a scattered subword and a consecutive subword. To be precise, a scattered subword of a finite word $u=u_1u_2\ldots u_n$ is a sequence of the form $u_{j_1}u_{j_2}\ldots u_{j_k}$ with $1 \leq j_1<j_2 <\cdots <j_k \leq n$, while a consecutive subword is of the form $u_iu_{i+1}\ldots u_j$ with $1 \leq i <j \leq j$. For example, consider $n=26$ and $w=10$. The binary expansion of $26$ is $11010$, so $s_w(26)=5$ and $e_w(26)=2$. In the literature, `scattered subword' can be abbreviated to `subword' while a consecutive subword can be called a factor, see \cite[Chapter 1, Chapter 6]{lothaire1}. In the context of this article, the sequence $((-1)^{s_w(n)})_{n \in \N}$ is called a subword-counting sequence and $((-1)^{e_w(n)})_{n \in \N}$ is called a factor-counting sequence.\\

Partial sums of factor-counting sequences have been widely studied since the 1960s, see~\cite{Rudin,Shapiro,allouche91,morton90,DavidW1989,Allouche99}. In particular, the sequence $((-1)^{e_{11}(n)})_{n \in \N}$ is the well-known $\pm 1$-Rudin-Shapiro sequence (\cite[Example 3.3.1]{allouche_shallit_2003}). Rudin and Shapiro proved separately in \cite{Rudin} and~\cite{Shapiro} that 
$$\max_{0 \leq \theta <1}\left|\sum_{n=0}^N(-1)^{e_{11}(n)}e^{2 \pi i n \theta}\right|=O(\sqrt{N}), N \geq 1.$$ The sequence $((-1)^{e_{1}(n)})_{n \in \N}$ is the $\pm 1$-Thue-Morse sequence (\cite[P. 15]{allouche_shallit_2003}) and a similar result involving this sequence was obtained by Gelfond in~\cite{Gelfond1968}. However, the partial sum of subword-counting sequences was less studied in the literature. It is proved by Lafrance et. al. in \cite{LAFRANCE} that $$\left|\sum_{n=0}^N(-1)^{s_{10}(n)}\right|= O(\sqrt{N}), N \geq 1,$$ while Allouche proved in \cite{Allouche2017} that $$\left|\sum_{n=0}^N(-1)^{s_{10}(n)}e^{2 \pi i n \theta}\right| \neq O(\sqrt{N}), N \geq 1,$$ for infinitely many $\theta$. In this article, we establish a framework to study the partial sum of general subword-counting sequences and characterize several classes of words $w$ satisfying the {\em Property $P$}:
\begin{equation}\label{growth} \left|\sum_{n=0}^N(-1)^{s_w(n)}\right|=O(N^{1-\epsilon}),\end{equation} for some $\epsilon >0$. 

The main results of this article are stated in \Cref{simple}, \Cref{one-run},\Cref{long-prefix} and \Cref{two-runs}. The article is organized as the follows: In \Cref{sec1} we recall some basic notion and definitions on words and on matrices and we introduce some basic lemmas concerning the subword-counting functions; in \Cref{sec2}, we illustrate our proof strategy by giving two particular examples and discuss the difference between the techniques applied to block-counting sequences and subword-counting sequences; in \Cref{sec3}, we introduce an efficient way to study the partial sums of the subword-counting sequences in using matrix computations and show how we construct these matrices; in \Cref{sec4}, we study the eigenvalues of these matrices and in \Cref{sec5} we study the orbits of some special vectors acted by these matrices; in \Cref{sec6} we prove the main results; in \Cref{sec7}, we conclude the article by reviewing the techniques used in this article and discuss some potential research direction.

\section{Notation and basic identities}\label{sec1}

Let $\{0,1\}^*$ be the set of all finite sequences over the alphabet $\{0,1\}$. The elements in this set are called {\em words}. In particular, the {\em empty word} also belongs to $\{0,1\}^*$, and is denoted by $\varepsilon$. For $w=w_1w_2\ldots w_\ell \in \{0,1\}^*$, $\ell$ is called the {\em length} of $w$. The length of $w\in\{0,1\}^*$ is also denoted by $|w|$. For $\ell$ any positive integer, define $\{0,1\}^\ell$ to be the set of length-$\ell$ words over $\{0,1\}$. A {\em prefix} (resp. {\em suffix}) of $w$ is a string of the form $w_1w_2\ldots w_t$ (resp. ${w_tw_{t+1}\ldots w_\ell}$) for some integer $t$ such that $1 \leq t \leq \ell$. A {\em factor} of $w$ is a string ${w_tw_{t+1}\ldots w_k}$ such that $1 \leq t \leq k \leq \ell$. Such a string is denoted by $w[t,k]$ for short. The {\em first} and {\em last} letters of $w$ refer respectively to the leftmost and rightmost letters of $w$. For any words $w,v \in \{0,1\}^*$, the {\em concatenation} of $w$ and $v$ is defined to be $wv=w_1w_2\ldots w_{|w|}v_1v_2\ldots v_{|v|}$. For $w \in \{0,1\}^*$ and $k\geq0$, let $w^k$ represent the $k$-th fold concatenation of $w$, i.e. $w^k=\underbrace{ww\ldots w}_{k \;\;\text{times}}$.\\

Let $\land$ and $\oplus$ denote the logical operations `and' and `exclusive or' respectively, i.e. for $a, b \in \{0,1\}$, $a \land b=1$ if $a=b=1$ and $a \land b=0$ otherwise; $a \oplus b=0$ if $a=b$ and $a \oplus b=1$ otherwise. For any $w,v \in \{0,1\}^*$ of the same length $\ell$, let us define $w \land v=(w_1\land v_1)(w_2\land v_2)\ldots (w_\ell\land v_\ell)$ and $w \oplus v=(w_1\oplus v_1)(w_2\oplus v_2)\ldots (w_\ell\oplus v_\ell)$. Likewise, let $\overline{\phantom{x}}$ denote the logical operation `not', i.e. $\overline{0}=1$ and $\overline{1}=0$. For any $w \in \{0,1\}^*$, define $\overline{w}=\overline{w_1}\overline{w_2}\ldots\overline{w_{|w|}}$.\\

Let $M_k(\C)$ denote the complex vector space of $k\times k$ matrices with complex entries. For $M\in M_k(\C)$, let $\|M\|$ denote the operator norm of $M$, i.e.
$$\|M\|:=\sup_{\substack{v\in\C^k\\\|v\|=1}}\|Mv\|,$$
where $\|v\|$ denotes the usual Euclidean norm of $v\in\C^k$.\\

Recall that, for any integer $n\geq0$, $s_w(n)$ denotes the number of occurrences of $w$ as a subword of the {\em ordinary binary expansion} of $n$. By ordinary binary expansion, we mean the binary expansion without leading zeros for all nonzero integers. By convention, let $s_0(0)=1$ and $s_\varepsilon(n)=1\,\,\forall\,n\geq0$. For $w\in\{0,1\}^\ast$ and $n\geq 1$, one has
\begin{align}
\begin{split}\label{w0}
s_{w0}(2n)&=s_{w0}(n)+s_w(n)\\
s_{w0}(2n+1)&=s_{w0}(n)\\
\end{split}\\
\begin{split}\label{w1}
s_{w1}(2n)&=s_{w1}(n)\\
s_{w1}(2n+1)&=s_{w1}(n)+s_{w}(n)\\
\end{split}
\end{align}
It is important to note that these recurrences do not always hold for $n=0$.

\section{Two examples}\label{sec2}

In this section, we introduce the general method to estimate the growth rate of $\sum\limits_{n=0}^N(-1)^{s_w(n)}$ by giving two examples.

\begin{exa}\label{ex-01}
Let us check whether $01$ satisfies \textit{Property P}. Hence, set $S_N=\sum\limits_{n=0}^N(-1)^{s_{01}(n)}$. We also introduce $T_N=\sum\limits_{n=0}^{N}(-1)^{s_{01}(n)+s_{0}(n)}$.
One has
\begin{align}\label{S_2N+1}
S_{2N+1}&=\sum_{n=0}^{2N+1}(-1)^{s_{01}(n)}\nonumber\\
&=(-1)^{s_{01}(0)}+(-1)^{s_{01}(1)}+\sum_{n=1}^{N}(-1)^{s_{01}(2n)}+\sum_{n=1}^{N}(-1)^{s_{01}(2n+1)}\nonumber\\
&=2+\sum_{n=1}^{N}(-1)^{s_{01}(n)}+\sum_{n=1}^{N}(-1)^{s_{01}(n)+s_{0}(n)}\text{ (using \eqref{w1})}\nonumber\\
&=2+S_N+T_N.
\end{align}
\begin{align}\label{T_2N+1}
T_{2N+1}&=\sum_{n=0}^{2N+1}(-1)^{s_{01}(n)+s_0(n)}\nonumber\\
&=(-1)^{s_{01}(0)+s_0(0)}+(-1)^{s_{01}(1)+s_0(1)}+\sum_{n=1}^{N}(-1)^{s_{01}(2n)+s_0(2n)}\nonumber\\
&\quad+\sum_{n=1}^{N}(-1)^{s_{01}(2n+1)+s_0(2n+1)}\nonumber\\
&=\sum_{n=1}^{N}(-1)^{s_{01}(n)+s_0(n)+s_\varepsilon(n)}+\sum_{n=1}^{N}(-1)^{s_{01}(n)+2s_{0}(n)}\text{ (using \eqref{w0} and \eqref{w1})}\nonumber\\
&=-2-T_N+S_N.
\end{align}
Adding \eqref{S_2N+1} and \eqref{T_2N+1} yields
$$S_{2N+1}+T_{2N+1}=2S_N,$$
and consequently, 
$$S_{4N+3}=2+2S_N\text{ (using \eqref{S_2N+1})}.$$
Since $|S_{4N+i}-S_{4N+3}|\leq 3$ for $0\leq i<4$, the above yields
\begin{align*}
|S_{N}|\leq 2\left|S_{\floor{\frac{N}{4}}}\right|+5,\,\,\forall\,N\geq0.
\end{align*}
Hence, we obtain
$$S_N= O\left(N^{\frac{1}{2}}\right).$$
A similar argument looking at the difference of \eqref{S_2N+1} and \eqref{T_2N+1} yields
$$T_N= O\left(N^{\frac{1}{2}}\right).$$
\end{exa}

\begin{rmk}
The splitting techniques used in \Cref{ex-01} has also appeared in the study of the partial sums of the form $\sum_{n=0}^N(-1)^{e_w(n)}$, see, for example, \cite[Theorem 3.3.2]{allouche_shallit_2003}. 
However, the subword-counting case is more complicated than the block-counting case. This is because in the block-counting case, the `new' summation obtained after a `splitting' only involves a block-counting sequence, but in the subword-counting case it involves a product of several (in this case, two) subword-counting sequences. Thus, a new technique is required, and in this vein we illustrate our general method with a more complicated example.
\end{rmk}



\begin{exa}\label{ex-011}
Now, we check whether $011$ satisfies \textit{Property P}. Hence, set $S_N=\sum\limits_{n=0}^N(-1)^{s_{011}(n)}$. As before, we introduce the following.
\begin{align*}
T_N&=\sum_{n=0}^{N}(-1)^{s_{011}(n)+s_{01}(n)};\\
U_N&=\sum_{n=0}^N(-1)^{s_{011}(n)+s_0(n)};\\
V_N&=\sum_{n=0}^N(-1)^{s_{011}(n)+s_{01}(n)+s_0(n)}.
\end{align*}
Using \eqref{w0} and \eqref{w1} as in \Cref{ex-01}, one has
\begin{align*}
S_{2N+1}&=\sum_{n=0}^{2N+1}(-1)^{s_{011}(n)}\\
&=(-1)^{s_{011}(0)}+(-1)^{s_{011}(1)}+\sum_{n=1}^{N}(-1)^{s_{011}(n)}+\sum_{n=1}^{N}(-1)^{s_{011}(n)+s_{01}(n)}\\
&=S_N+T_N;\\
T_{2N+1}&=2+\sum_{n=1}^N(-1)^{s_{011}(n)+s_{01}(n)}+\sum_{n=1}^N(-1)^{s_{011}(n)+2s_{01}(n)+s_0(n)}\\
&=2+T_N+U_N;\\
U_{2N+1}&=\sum_{n=1}^N(-1)^{s_{011}(n)+s_0(n)+s_\varepsilon(n)}+\sum_{n=1}^N(-1)^{s_{011}(n)+s_{01}(n)+s_0(n)}\\
&=-U_N+V_N;\\
V_{2N+1}&=\sum_{n=1}^N(-1)^{s_{011}(n)+s_{01}(n)+s_0(n)+s_\varepsilon(n)}+\sum_{n=1}^N(-1)^{s_{011}(n)+2s_{01}(n)2+s_0(n)}\\
&=-2-V_N+S_N.
\end{align*}
Since there are many equations to work with for obtaining a recurrence for $S_N$, it is helpful to represent the above equations in matrix-vector form. Hence, we define
$$c=\begin{bmatrix}
0\\
2\\
0\\
-2
\end{bmatrix}, M=\begin{bmatrix}
1 & 1 & 0 & 0\\
0 & 1 & 1 & 0\\
0 & 0 & -1 & 1\\
1 & 0 & 0 & -1
\end{bmatrix}, V_N=\begin{bmatrix}
S_{N}\\
T_{N}\\
U_{N}\\
V_{N}
\end{bmatrix}.$$
The above equations can now be stated succinctly as
\begin{align}
v_{2N+1}=c+Mv_N.\label{matrix-vector-011}
\end{align}
Since $\|v_{2N+i}-v_{2N+1}\|$ for $i\in\{0,1\}$ is bounded, \eqref{matrix-vector-011} yields
\begin{align}\label{v-growth-011}
\left\|v_{N}\right\|\leq c'+\left\|Mv_{\floor{\frac{N}{2}}}\right\|,
\end{align}
for some constant $c'$. To show that $011$ satisfies \textit{Property P}, it now suffices to show that $\|v_N\|\in O(N^{1-\epsilon})$ for some $\epsilon>0$. In light of \eqref{v-growth-011}, we may instead show that $\|M^n\|\in O(2^{n(1-\epsilon)})$, which happens if and only if all eigenvalues of $M$ have absolute value less than $2$. To analyse the eigenvalues of $M$, observe that
$$M=\underbrace{\begin{bmatrix}
1 & 0 & 0 & 0\\
0 & 1 & 0 & 0\\
0 & 0 & -1 & 0\\
0 & 0 & 0 & -1
\end{bmatrix}}_A+\underbrace{\begin{bmatrix}
0 & 1 & 0 & 0\\
0 & 0 & 1 & 0\\
0 & 0 & 0 & 1\\
1 & 0 & 0 & 0
\end{bmatrix}}_B.$$
Suppose $\lambda\in\C$ is an eigenvalue of $M$ with $|\lambda|\geq2$, and let $v\in\C^4\setminus \{0\}$ be the corresponding eigenvector. One has
\begin{align}\label{f_0}\lambda v=Mv=Av+Bv.\end{align}
Since $\|Av\|=\|Bv\|=\|v\|$ for all $v$, \eqref{f_0} implies that $|\lambda|\leq 2$ with equality if and only if $Av=Bv=\frac{\lambda}{2}v$. Hence $|\lambda|=2$ implies that $v$ is also an eigenvector of $A$ and $B$ with eigenvalue $\frac{\lambda}{2}$. In particular, all entries of $v$ must be non-zero since $B$ is a cyclic permutation matrix. However this implies that $v$ cannot be an eigenvector of $A$. This is a contradiction, so we conclude that all eigenvalues of $M$ have absolute value less than $2$, and consequently $011$ satisfies \textit{Property P}.\\

Additionally, if $\lambda$ is the eigenvalue of $M$ with largest absolute value, then one has
$$\|v_N\|=O\left(N^{1-\epsilon}\right)$$
for all $\epsilon<1-\log_2|\lambda|$. In this specific example the Jordan Canonical Form of $M$ is
$$\begin{bmatrix}
    0 & 1 & 0 & 0\\
    0 & 0 & 0 & 0\\
    0 & 0 & \sqrt{2} & 0\\
    0 & 0 & 0 & -\sqrt{2}
\end{bmatrix},$$
so in fact we obtain
$$\|v_N\|=O\left(N^{\frac{1}{2}}\right).$$
This shows that $S_N,T_N,U_N,V_N\in O(N^{\frac{1}{2}})$.
\end{exa}


\section{Constructing the associated matrix}\label{sec3}
In this section, we will construct the matrix that plays a role analogous to that of $M$ in \Cref{ex-011} in a general setup. First, we introduce some notation for summations of subword-counting sequences and their products. For $\ell\geq1$ and $u,w\in\{0,1\}^\ell$, define 
\begin{align*}
\begin{bmatrix}
w\\
u
\end{bmatrix}: \N \to \{0,1\};n\mapsto&\sum_{\substack{i=1\\\text{with } u_i=1}}^\ell s_{u_1\ldots u_i}(n)\pmod2\\
=&\quad\sum_{i=1}^\ell u_i\cdot s_{u_1\ldots u_i}(n)\pmod2.
\end{align*}

\begin{exa}\label{ex-brackets}
The above definition yields
$$\begin{bmatrix}
abcd\\
0101
\end{bmatrix}(n)=s_{ab}(n)+s_{abcd}(n)\pmod2.$$
\end{exa}
With $\ell,u,w$ as before, define functions $S_0(w),S_1(w):\{0,1\}^\ell\to\{0,1\}^\ell$ and $T_0(w),T_1(w):\{0,1\}^\ell\to\{0,1\}$ as follows, where $u'_i,u''_i,t',t''$ defined below.
\begin{align*}
S_0(w)(u)&:=u'_1\ldots u'_\ell\\
S_1(w)(u)&:=u''_1\ldots u''_\ell\\
T_0(w)(u)&:=t'\\
T_1(w)(u)&:=t''
\end{align*}
$u'_i,u''_i,t',t''$ are defined as follows.
\begin{align}
\begin{split}\label{logical-defn}
t'u'_1\ldots u'_\ell=\left(\overline{w}0\land u0\right)\oplus 0u;\\
t''u''_1\ldots u''_\ell=\left(w0\land u0\right)\oplus 0u.
\end{split}
\end{align}
A compact version of this definition is as follows, where $a\in\{0,1\}$.
\begin{equation}\label{combined}
T_a(w)(u)\,S_a(w)(u):=\left[\left(\overline{a}^\ell\oplus w\right)\hspace{-2pt}0\land u0\right]\oplus 0u.
\end{equation}
\begin{rmk}\label{symmetry}
It is easy to see that $S_{a}(w)=S_{\overline{a}}(\overline{w})$ and $T_{a}(w)=T_{\overline{a}}(\overline{w})$.
\end{rmk}

\begin{lem}
One has
\begin{align}
\begin{bmatrix}
w\\
u
\end{bmatrix}(2n)&=\begin{bmatrix}
w\\
S_0(w)(u)
\end{bmatrix}(n)+T_0(w)(u)\pmod2\text{ (for }n\geq1)\label{S_0-sq};\\
\begin{bmatrix}
w\\
u
\end{bmatrix}(2n+1)&=\begin{bmatrix}
w\\
S_1(w)(u)
\end{bmatrix}(n)+T_1(w)(u)\pmod2\text{ (for }n\geq1)\label{S_1-sq}.
\end{align}
\end{lem}

\begin{proof}
We prove \eqref{S_0-sq}, and \eqref{S_1-sq} follows analogously. From the definition,
\begin{align*}
\begin{bmatrix}
w\\
u
\end{bmatrix}(2n)&=\sum_{\substack{i=1\\\text{with } u_i=1}}^\ell s_{u_1\ldots u_i}(2n)\pmod2\\
&=\sum_{\substack{i=1\\\text{with } u_i=1,\\w_i=1}}^\ell s_{u_1\ldots u_i}(2n)+\sum_{\substack{i=1\\\text{with } u_i=1,\\w_i=0}}^\ell s_{u_1\ldots u_i}(2n)\pmod2\\
&=\sum_{\substack{i=1\\\text{with } u_i=1,\\w_i=1}}^\ell s_{u_1\ldots u_i}(n)+\sum_{\substack{i=1\\\text{with } u_i=1,\\w_i=0}}^\ell s_{u_1\ldots u_{i-1}}(n)\\
&\quad+\sum_{\substack{i=1\\\text{with } u_i=1,\\w_i=0}}^\ell s_{u_1\ldots u_{i}}(n)\pmod2\text{ (using \eqref{w0} and \eqref{w1})}\\
&=\sum_{\substack{i=1\\\text{with } u_i=1}}^\ell s_{u_1\ldots u_i}(n)+\sum_{\substack{i=2\\\text{with } u_i=1,\\w_i=0}}^\ell s_{u_1\ldots u_{i-1}}(n)+\alpha\pmod2\\
&=\begin{bmatrix}
w\\
u
\end{bmatrix}(n)+\begin{bmatrix}
w\\
u'
\end{bmatrix}(n)+\alpha\pmod2,
\end{align*}
where we interpret $u_1\ldots u_{i-1}$ to be $\varepsilon$ for $i=1$, and $\alpha$ and $u'$ are as follows. We set $\alpha=1$ if $w_1=0$ and $u_1=1$, and $\alpha=0$ otherwise. The word $u'=u'_1\ldots u'_\ell\in\{0,1\}^\ell$ is defined as $u'_i=1$ if and only if $1\leq i<\ell$ and $u_{i+1}=1$ and $w_{i+1}=0$. Thus, $u'=(\overline{w}0 \land u0)[2,\ell+1]$ and $\alpha=T_0(w)(u)$. Now \eqref{S_0-sq} follows. 
\end{proof}

Henceforth, let $u,w$ denote words on $\{0,1\}$ with lengths $|u|=|w|=\ell\geq 2$ and $h=h_1\ldots h_r\in\{0,1\}^*$, unless explicitly mentioned otherwise. Define $S_h(w)=S_{h_1}(w)\circ\ldots\circ S_{h_r}(w)$, i.e. the composition of the functions $S_{h_i}(w)$, and $T_{h}(w)=T_{h_1}(w)+\ldots+T_{h_r}(w)$. We take $S_\varepsilon(w)$ to be the identity map and $T_{\varepsilon}(w)=0$.

\begin{lem}\label{observe}
\begin{enumerate}[(a)]
\item $S_h(w)(u\oplus u')=S_h(w)(u)\oplus S_h(w)(u')$ and $T_h(w)(u\oplus u')=T_h(w)(u)\oplus T_h(w)(u')$.

\item Let $u',w'$ be suffixes of $u,w$ respectively with $|u'|=|w'|\geq1$. Then $S_h(w')(u')$ is a suffix of $S_h(w)(u)$. 
\end{enumerate}
\end{lem}

\begin{proof}
It suffices to prove both claims for the case when $h\in\{0,1\}$ is a letter. In this case, the claims follow at once from \eqref{combined}. 
\end{proof}

\begin{lem}\label{permutations}
$S_h(w):\{0,1\}^\ell\to\{0,1\}^\ell$ is a bijection.
\end{lem}

\begin{proof}
It suffices to prove the claim for the case when $h\in\{0,1\}$. Notice that for any $a,x\in\{0,1\}$,
\begin{align}\label{solve-xor}
x=a\oplus (a\oplus x).
\end{align}
In particular, $x$ is uniquely determined once $a$ and $a\oplus x$ are given. Now, fix $u'=u'_1\ldots u'_\ell\in\{0,1\}^{\ell}$. We will solve for a unique word $u=u_1\ldots u_\ell\in\{0,1\}^\ell$ such that $S_h(w)(u)=u'$. By \eqref{combined}, we have $u'_\ell=(0\land0)\oplus u_\ell$ and
$$u'_{i}=[(\overline{h}\oplus w_{i+1})\land u_{i+1}]\oplus u_i\text{ for $1\leq i\leq \ell-1$}.$$
Hence \eqref{solve-xor} yields
\begin{align*}
u_\ell&=u'_\ell;\\
u_i&=[(\overline{h}\oplus w_{i+1})\land u_{i+1}]\oplus u'_i\text{ for }1\leq i\leq \ell-1.
\end{align*}
Hence, $u_i$ can be solved for recursively. 
\end{proof}

\newcommand{\calO}{\mathcal{O}}

Let $\calO_w(u)$ be the orbit of $u$ under the group of permuations of $\{0,1\}^\ell$ generated by $S_0(w)$ and $S_1(w)$. Explicitly,
$$\calO_w(u)=\left\{S_h(w)(u)\mid h\in\{0,1\}^\ast\right\}.$$

We give special emphasis to the case $u=0^{\ell-1}1$ by defining $\calO_{w}:=\calO_w(0^{\ell-1}1)$.

\begin{lem}\label{last-letter}
If $u$ has suffix $10^k$ ($k\geq0$) then every word in $\calO_w(u)$ has suffix $10^k$. In particular, every element of $\calO_{w}$ has $1$ as a suffix. 
\end{lem}

\begin{proof}
Let $w'$ be the suffix of $w$ of length $k+1$. Observe that $S_0(w')(10^k)=S_1(w')(10^k)=10^k$. The claim now follows from (a) of \Cref{observe}. 
\end{proof}

\begin{cor}\label{matrix-size-bound}
$|\calO_w(u)|\leq 2^{\ell-1}$.
\end{cor}

\begin{proof}
This is a direct consequence of \Cref{last-letter}. 
\end{proof}

\begin{lem}\label{prefix}
Suppose $u$ and $h$ satisfy the following.
\begin{enumerate}[(a)]
    \item  $u$ has prefix $0^{k-1}$ for some $k \leq \ell$.
\item $|h|<k$.
\end{enumerate}
Then $S_h(w)(u)$ has prefix $0^{k-|h|-1}$. In particular, if $|h|<\ell$, then $S_h(w)(0^{\ell-1}1)$ has prefix $0^{\ell-|h|-1}$.
\end{lem}

\begin{proof}
It suffices to prove the claim when $h\in\{0,1\}$ is a letter, since the general claim follows by induction on $|h|$. For every integer $i$ such that $1 \leq i \leq \ell$, one has
\begin{align}\label{s_w}
S_{w_1}(w)(10^{\ell-1})&=S_{\overline{w_1}}(w)(10^{\ell-1})=10^{\ell-1};\nonumber\\
S_{w_i}(w)(0^{i-1}10^{\ell-i})&=0^{i-2}110^{\ell-i}\;\; \text{for }2 \leq i \leq \ell;\nonumber\\
S_{\overline{w_i}}(w)(0^{i-1}10^{\ell-i})&=0^{i-1}10^{\ell-i}\;\; \text{for }2 \leq i \leq \ell.
\end{align}
Thus, using (a) of \Cref{observe} yields the following.
\begin{align*}
S_h(w)(u)&=S_h(w)(0^{k-1}u_{k}0^{\ell-k}\oplus0^{k}u_{k+1}0^{\ell-k-1}\oplus \ldots \oplus 0^{\ell-1}u_{\ell})\\
&=S_h(w)(0^{k-1}u_{k}0^{\ell-k})\oplus S_h(w)(0^{k}u_{k+1}0^{\ell-k-1})\oplus \ldots \oplus S_h(w)(0^{\ell-1}u_{\ell}).
\end{align*}
Each term in the above has prefix $0^{k-1}$ (by \eqref{s_w}), so the claim follows. 
\end{proof}

\begin{cor}\label{matrix-size-bound-1}
$|\calO_{w}|\geq \ell$.
\end{cor}

\begin{proof}
Set $z_{0}=0^{\ell-1}1$ and $z_i=S_{w_{\ell-i+1}\ldots w_\ell}(w)(0^{\ell-1}1)$ for $1\leq i<\ell$. From \eqref{s_w}, one can prove by induction on $i$ that $z_i$ has prefix $0^{\ell-i-1}1$ for $0\leq i<\ell$.  Hence the $z_i$'s are $\ell$ distinct elements of $\calO_{w}$. 
\end{proof}

We now define three $|\calO_w(u)|\times|\calO_w(u)|$ matrices $M_0(w,u), M_1(w,u)$ and $M(w,u)$ whose rows and columns are indexed by elements of $\calO_w(u)$. For $u',u''\in\calO_w(u)$, we will write $M(w,u)[u',u'']$ to refer to the entry of $M(w,u)$ indexed by the pair $(u',u'')\in\calO_w(u)\times\calO_w(u)$.
\begin{align*}
M_0(w,u)[u',u'']&:=\begin{cases}
(-1)^{T_0(w)(u')} & \text{if }S_0(w)(u')=u'';\\
0 & \text{otherwise}.
\end{cases}\\
M_1(w,u)[u',u'']&:=\begin{cases}
(-1)^{T_1(w)(u')} & \text{if }S_1(w)(u')=u'';\\
0 & \text{otherwise}.
\end{cases}\\
M(w,u)&:=M_0(w,u)+M_1(w,u).
\end{align*}
Following are some simple but important observations regarding these matrices.

\begin{cor}\label{basic-props}
The following hold.
\begin{enumerate}[(a)]
\item $M_0(w,u)$ and $M_1(w,u)$ each have exactly one non-zero entry in each row and each column.
\item $\|M_0(w,u)v\|=\|M_1(w,u)v\|=\|v\|\,\,\forall\,v\in\C^{|\calO_w(u)|}$.
\end{enumerate}
\end{cor}

\begin{proof}
\begin{enumerate}[(a)]
\item Follows from \Cref{permutations} and the fact that $\calO_w(u)$ is closed under applications of $S_0(w)$ and $S_1(w)$.\\

\item Follows from (a). \qedhere
\end{enumerate}
\end{proof}

For $n\geq0$, let $v(w,u)(n)$ denote the $|\calO_w(u)|$-dimensional column vector, with entries indexed by $\calO_w(u)$, such that the entry in the $u'$-th row is
$$(-1)^{\begin{bmatrix}
w\\
u'
\end{bmatrix}(n)}.$$
Let $V(w,u)(N)=\sum\limits_{n=0}^N v(w,u)(n)$. As before, we place emphasis on the case of $u=0^{\ell-1}1$ by defining $M_0(w)=M_0(w,0^{\ell-1}1)$, $M_1(w)=M_1(w,0^{\ell-1}1)$, $M(w)=M(w,0^{\ell-1}1)$, $v(w)=v(w,0^{\ell-1}1)$, and $V(w)=V(w,0^{\ell-1}1)$.

\begin{exa}
Let $w=011$ as in \Cref{ex-011} and let $u=001$. One can check $\calO_w=\calO_w(u)=\{001,011, 101, 111\}$. Viewing $001,011,101,111$ as an ordered basis of $\C^4$ (under the convention of indexation by $\calO_w(u)$), the matrix
$$M(011,001)=\begin{bmatrix}
1 & 1 & 0 & 0\\
0 & 1 & 1 & 0\\
0 & 0 & -1 & 1\\
1 & 0 & 0 & -1
\end{bmatrix}$$
is the same as $M$ from \Cref{ex-011}. Furthermore,
$$v(011,011)(n)=\begin{bmatrix}
(-1)^{s_{001}(n)}\\
(-1)^{s_{01}(n)+s_{001}(n)}\\
(-1)^{s_{0}(n)+s_{001}(n)}\\
(-1)^{s_{0}(n)+s_{01}(n)+s_{001}(n)}
\end{bmatrix},$$
so $V(011,001)(N)$ is precisely $v_N$ from \Cref{ex-011}.
\end{exa}


\begin{thm}\label{matrix-recurrence}
For $n\geq1$ and $i\in\{0,1\}$, we have
$$v(w,u)(2n+i)=M_i(w,u)\,v(w,u)(n).$$
Consequently, for $N\geq0$ we have
$$V(w,u)(2N+1)=M(w,u)\,V(w,u)(N)+c(w,u),$$
where $c(w,u)$ is given by
$$c(w,u):=\left(I-M(w,u)\right)v(w,u)(0)+v(w,u)(1),$$
with $I$ the identity matrix of size $|\calO_w(u)|\times|\calO_w(u)|$.
\end{thm}

\begin{proof}
For $n\geq1$ and $i\in\{0,1\}$, \eqref{S_0-sq} and \eqref{S_1-sq} yield
$$(-1)^{\begin{bmatrix}
w\\
u
\end{bmatrix}(2n+i)}=(-1)^{\begin{bmatrix}
w\\
S_i(w)(u)
\end{bmatrix}(n)}(-1)^{T_i(w)(u)}.$$
The first equation in the theorem now follows. To prove the second equation, we use the first equation as follows. For convenience, let $M=M(w,u)$, $M_0=M(w,u)$, $M_1=M(w,u)$, $V=V(w,u)$, and $v=v(w,u)$.
\begin{align*}
V(2N+1)&=\sum_{n=0}^{2N+1}v(n)\\
&=v(0)+v(1)+\sum_{n=1}^Nv(2n)+\sum_{n=1}^Nv(2n+1)\\
&=v(0)+v(1)+\sum_{n=1}^NM_0v(n)+\sum_{n=1}^NM_1v(n)\\
&=v(0)+v(1)+\sum_{n=1}^NMv(n)\\
&=v(0)+v(1)-Mv(0)+\sum_{n=0}^NMv(n)\\
&=MV(N)+(I-M)v(0)+v(1).\hspace{10pt}\qedhere
\end{align*}
\end{proof}

\begin{thm}\label{eval-growth}
If $|\lambda|<2$ for all eigenvalues $\lambda$ of $M(w,u)$, then there exists $\epsilon>0$ such that
$$\|V(w,u)(N)\|= O\left(N^{1-\epsilon}\right)$$
Consequently, for all $u'\in\calO_w(u)$,
$$\left|\sum_{n=0}^N(-1)^{\begin{bmatrix}
w\\
u'
\end{bmatrix}(n)}\right|= O\left(N^{1-\epsilon}\right).$$
In particular, if $u=0^{\ell-1}1$, then $w$ satisfies \emph{Property P}.
\end{thm}

\begin{proof}
Let $N_0=N$ and $L=\floor{\log_2(N)}+1$. Set $N_{i+1}=\floor{\frac{N_i}{2}}$ for $1\leq i\leq L$. Hence $N_L=0$ and for each $i<L$, we have $N_i=2N_{i+1}$ if $N_i$ is even and $N_i=2N_{i+1}+1$ if $N_i$ is odd. For $1\leq i\leq L$, let
$$c_i=V(w,u)(N_{i-1})-V(w,u)(2N_i+1).$$
Although $c_i$ depends on $N$, observe that $\|c_i\|\leq \sqrt{|\calO_w(u)|}$ since $c_i=0$ if $N_{i-1}$ is odd and all entries of $c_i$ are $\pm1$ if $N_{i-1}$ is even. Now, repeated applications of \Cref{matrix-recurrence} yield the following. As before, we set $M=M(w,u)$ et cetera.
\begin{align*}
V(N)&=V(N_0)\\
&=MV(N_1)+c+c_1\\
&=M^2V(N_2)+M(c+c_2)+(c+c_1)\\
&=\ldots\\
&=M^iV(N_i)+\sum_{r=0}^{i-1}M^r(c+c_{r+1})\\
&=\ldots\\
&=M^LV(N_L)+\sum_{r=0}^{L-1}M^r(c+c_{r+1})\\
&=M^Lv(0)+\sum_{r=0}^{L-1}M^r(c+c_{r+1}).
\end{align*}
Hence
$$\|V(N)\|\leq c'\sum_{r=0}^L\|M^r\|,$$
where $c'=\|c\|+\sqrt{|\calO_w(u)|}$. Now, recall that if $\lambda$ is an eigenvalue of $M$ of largest magnitude, then
$$\|M^n\|=O(|\lambda|^np(n)),$$
for some polynomial $p$. Hence, the above calculation yields
$$\|V(w)(N)\|=O\left(N^{\log_2|\lambda|}p\left(\log_2(N)\right)\right).$$
Since $|\lambda|<2$ by hypothesis, we may take $\epsilon<1-\log_2|\lambda|$. 
\end{proof}


In light of \Cref{eval-growth}, we consider a property for pairs $(w,u)$ which is analogous to {\em Property P}. A pair $(w,u)$ is said to satisfy \textit{Property Q} if there exists $\epsilon>0$ such that
\begin{equation}\label{growth-new}
\|V(w,u)(N)\|\in O\left(N^{1-\epsilon}\right).
\end{equation}
Hence, $w$ satisfies \textit{Property P} if $(w,0^{\ell-1}1)$ satisfies \textit{Property Q} (although the converse need not hold).

\section{Eigenvalues of $M(w,u)$}\label{sec4}

\begin{thm}\label{eigenval-2}
If $\lambda$ is an eigenvalue of $M(w,u)$ with $|\lambda|\geq2$ and corresponding eigenvector $v\in\C^{|\calO_w(u)|}-\{0\}$, then the following hold.
\begin{enumerate}[(a)]
\item $|\lambda|=2$ and $v$ is an eigenvector of both $M_0(w,u)$ and $M_1(w,u)$ with eigenvalue $\frac{\lambda}{2}$.
\item All entries of $v$ are non-zero.
\item If $S_a(w)(u)=u$ and $T_a(w)(u)=0$ for some $a\in\{0,1\}$, then $\lambda=2$ and $v=cx$, for some $c\in\C$ and $x\in\{1,-1\}^{|\calO_w(u)|}$.
\end{enumerate}
\end{thm}

\begin{rmk}\label{imp-remark}
For $u=0^{\ell-1}1$, the hypothesis of (c) is satisfied with $a=\overline{w_k}$.
\end{rmk}

\begin{proof}
\begin{enumerate}[(a)]
\item From definitions, $M_0(w,u)\,v+M_1(w,u)\,v=\lambda v$. Since $|\lambda|\geq2$, one has $\|M_0(w,u)\,v+M_1(w,u)\,v\|\geq2\|v\|$.
However, from \Cref{basic-props}, $\|M_0(w,u)\,v\|=\|M_1(w,u)\,v\|=\|v\|$. Thus, $|\lambda|=2$ and $M_0(w,u)\,v$ and $M_1(w,u)\,v$ are positive multiples of each-other. Hence,
\begin{align}\label{eigenvals}
M_0(w,u)\,v=M_1(w,u)\,v=\frac{\lambda}{2}v.
\end{align}

\item Let $c\in\C$ be the the $u$-th entry of $v$. Pick arbitrary $u'\in\calO_w(u)$ and let $c'$ be the $u'$-th entry of $v$. By definition of $\calO_w(u)$, there exists $h=h_1\ldots h_r\in\{0,1\}^\ast$ such that $S_{h}(w)(u)=u'$. The $u$-th row of the matrix
$$M_{h_r}(w,u)\ldots M_{h_1}(w,u)$$
has all entries 0 except for that in the $S_{h}(w)(u)$-th column, which is $(-1)^{T_{h}(w)(u)}$. Since $S_{h}(w)(u)=u'$, this means that the entry in the $u$-th row of the vector
$$M_{h_r}(w,u)\ldots M_{h_1}(w,u)\,v$$
is $(-1)^{T_{h}(w)(u)}c'$. Hence, \eqref{eigenvals} yields that
\begin{equation*}
(-1)^{T_h(w)(u)}c'=\left(\frac{\lambda}{2}\right)^r\h{-2}c.
\end{equation*}
In particular we have $|c|=|c'|$ (since $|\lambda|=2$), and since $u'$ was chosen arbitrarily and $v\neq0$ we see that all entries of $v$ are non-zero.



\item By the hypothesis on $u$, the $(u,u)$-th entry of $M_{a}(w,u)$ is $1$. Consequently, the $u$-th entry of $M_{a}(w,u)\,v$ is $c$. By \eqref{eigenvals} this yields $\frac{\lambda}{2}=1$ (since $c\neq 0$).\qedhere
\end{enumerate}
\end{proof}

A $k\times k$ matrix $M$ with entries in $\{0,1,-1\}$ is called a {\em signed permutation matrix} if it has exactly one non-zero entry in each row and each column. Let $\sigma_M$ denote the permutation of $[k]$ such that the entry of $M$ in the position $(i,\sigma_M(i))$ is non-zero for each $i\in[k]$. If the entries of $M$ are indexed by tuples with entries from some set $S$ different from $[k]$, then we will take $\sigma_M$ to be a permutation of $S$ instead. In particular, for $M=M_0(w,u)$ or $M_1(w,u)$ we will take $\sigma_M$ to be a permutation of $\calO_w(u)$. Say that $M$ is cyclic whenever $\sigma_M$ is cyclic.

\begin{lem}\label{sign-cyclic-eigenval}
If $M$ is a cyclic signed permutation matrix, then all eigenvalues of $M$ are distinct. Furthermore, $M$ has eigenvalue $1$ if and only if $M$ has an even number of entries which are $-1$.
\end{lem}

\newcommand\hash{\text{\ttfamily\#}}

\begin{proof}
By conjugating $M$ by a permutation matrix if necessary, we may assume that $\sigma_M=(1\,2\ldots\,k)$ in cyclic notation. Let $R\subset [k]$ be the set of indices $r$ such that the $r$-th row of $M$ has a $-1$. We will show that all the $k$-th roots of $(-1)^{|R|}$ are eigenvalues of $M$, and since a $k\times k$ matrix can have at most $k$ distinct eigenvalues, it will follow that these are all the eigenvalues of $M$. From here the lemma easily follows.\\

Let $\omega$ be a primitive $2k$-th root of $1$. For each $p\in[k]$ let $v_p$ be the $k$-dimensional column vector whose $q$-th entry is given by
$$v_p(q):=(-1)^{\hash\{r\in R\,\mid\,r\geq q\}}\omega^{(2p+|R|)q}.$$
Hence, for $q<k$, the $q$-th entry of $Mv_p$ is given by
\begin{align*}
(Mv_p)(q)&=\begin{cases}
v_p(q+1) & q\notin R\\
-v_p(q+1) & q\in R
\end{cases}\\
&=\begin{cases}
(-1)^{\hash\{r\in R\,\mid\,r\geq q+1\}}\omega^{(2p+|R|)(q+1)} & q\notin R\\
(-1)^{\hash\{r\in R\,\mid\,r\geq q+1\}-1}\omega^{(2p+|R|)(q+1)} & q\in R
\end{cases}\\
&=(-1)^{\hash\{r\in R\,\mid\,r\geq q\}}\omega^{(2p+|R|)(q+1)}\\
&=\omega^{(2p+|R|)}v_p(q).
\end{align*}
The $k$-th entry of of $Mv_p$ is given by
\begin{align*}
(Mv_p)(k)&=\begin{cases}
v_p(1) & k\notin R\\
-v_p(1) & k\in R
\end{cases}\\
&=\begin{cases}
(-1)^{\hash\{r\in R\,\mid\,r\geq 1\}}\omega^{2p+|R|} & k\notin R\\
(-1)^{\hash\{r\in R\,\mid\,r\geq 1\}+1}\omega^{2p+|R|} & k\in R
\end{cases}\\
&=(-1)^{|R|+\hash\{r\in R\,\mid\,r\geq k\}}\omega^{2p+|R|}\\
&=(-1)^{\hash\{r\in R\,\mid\,r\geq k\}}\omega^{(2p+|R|)(k+1)}\\
&=\omega^{2p+|R|}v_p(k).
\end{align*}
Hence, $\omega^{2p+|R|}$ is an eigenvalue of $M$ with eigenvector $v_p$ for each $p\in[k]$. As $p$ varies over $[k]$, $\omega^{2p+|R|}$ varies over all $k$-th roots of $(-1)^{|R|}$. Hence, the claim follows. 
\end{proof}

Suppose $M$ is a signed permutation matrix and the cycle $\pi=(i_1\ldots i_r)$ appears in the cycle-decomposition of $\sigma_M$. The number of $-1$'s in $\pi$ (with respect to $M$) is defined to be the quantity $\hash\{j\mid\text{the }i_j\text{-th row of }M\text{ has a }-1\}$. Note that this is the same as $\hash\{j\mid\text{the }i_j\text{-th column of }M\text{ has a }-1\}$.

\begin{thm}\label{sign-perm-eigenval}
For $M$ a signed permutation matrix, the following are equivalent.
\begin{enumerate}[(a)]
    \item There exists an eigenvector of $M$ with eigenvalue $1$ and all entries non-zero.
    \item The number of $-1$'s in every cycle in $\sigma_M$ is even.
\end{enumerate}
\end{thm}

\begin{proof}
There exists a permutation matrix $X$ such that $XMX^{-1}$ is a block matrix,
$$XMX^{-1}=\begin{bmatrix}
M_1 & 0 &  & 0\\
0 & M_2 & \ldots & 0\\
0 & 0 &  & M_r
\end{bmatrix},$$
where each $M_i$ is a cyclic signed permutation matrix.\\

If all cycles of $\sigma_M$ have an even number of $-1$'s (with respect to $M$), then each $M_i$ has an even number of $-1$'s. Hence, each $M_i$ has an eigenvector with eigenvalue $1$ and all entries non-zero (by \Cref{sign-cyclic-eigenval}). Stacking these eigenvectors, we obtain an eigenvector $v$ of $XMX^{-1}$ with eigenvalue $1$ and all entries non-zero. Now $X^{-1}v$ is the desired eigenvector of $M$.\\

Conversely, if some cycle of $\sigma_M$ has an odd number of $-1$'s (with respect to $M$), then there exists $i$ such that $M_i$ has an odd number of $-1$'s. Hence, $M_i$ does not have eigenvalue $1$ (by \Cref{sign-cyclic-eigenval}). Consequently, any eigenvector of $XMX^{-1}$ with eigenvalue $1$ cannot have non-zero entries in the rows corresponding to $M_i$. Thus, $M$ cannot have an eigenvector with eigenvalue $1$ and all entries non-zero. 
\end{proof}

From \Cref{basic-props}, we know that $M_a(w,u)$, $a\in\{0,1\}$, is a signed permutation matrix. Furthermore, $\sigma_{M_a(w,u)}=S_a(w)|_{\calO_w(u)}$, the restriction of $S_a(w)$ to $\calO_w(u)$. Hence, we obtain the following sufficient conditions for \textit{Properties P} and \emph{Q} to be satisfied.

\begin{cor}\label{odd-minus-1s}
Suppose $a,b\in\{0,1\}$ such that $S_a(w)(u)=u$, $T_a(w)(u)=0$, and some cycle of $\sigma_{M_b(w,u)}$ has an odd number of $-1$'s. Then $(w,u)$ satisfies {\em Property Q}. In particular, if the number of $-1$'s in some cycle of $\sigma_{M_0(w)}$ or $\sigma_{M_1(w)}$ is odd, then $w$ satisfies {\em Property P}.
\end{cor}

\begin{proof}
It is clear that the second statement follows from the first by \Cref{imp-remark}, so we will prove the first statement. By \Cref{sign-perm-eigenval}, $M_b(w,u)$ does not have an eigenvector with eigenvalue $1$ and all entries non-zero. Thus, by \Cref{eigenval-2}, $M(w,u)$ cannot have an eigenvalue $\lambda$ with $|\lambda|\geq 2$. The desired result now follows from \Cref{eval-growth}. 
\end{proof}

From the definition \eqref{logical-defn}, we see that $T_a(w)(u)=1$ if and only if $w$ has first letter $a$ and $u$ has first letter $1$. In other words, for $u'\in\calO_w(u)$, the non-zero entry in the $u'$-th row of $M_a(w,u)$ is $-1$ if and only if $w$ has first letter $a$ and $u'$ has first letter $1$. Hence, the number of $-1$'s in any cycle $C$ of $\sigma_{M_a(w,u)}$ is given by
\begin{itemize}
\item $0$, if $a$ is not the first letter of $w$.
\item the number of words in $C$ with first letter $1$, if $a$ is the first letter of $w$.
\end{itemize}
Since $\sigma_{M_a(w,u)}=S_a(w)|_{\calO_w(u)}$, \Cref{odd-minus-1s} can now be stated in the following equivalent way.

\begin{thm}\label{imp-suff-cond}
Suppose $a,b\in\{0,1\}$ such that $S_a(w)(u)=u$, $T_a(w)(u)=0$, and some cycle of $S_b(w)|_{\calO_w(u)}$ has an odd number of words with first letter $1$. Then $(w,u)$ satisfies {\em Property Q}. In particular, if some cycle of $S_0(w)|_{\calO_w}$ or $S_1(w)|_{\calO_w}$ has an odd number of words with first letter $1$, then $w$ has {\em Property P}.
\end{thm}

The above is essentially at the core of the proofs of all our main results.

\section{Cycles of $S_0(w)|_{\calO_w(u)}$ and $S_1(w)|_{\calO_w(u)}$}\label{sec5}

For $a\in\{0,1\}$, let $\Lambda_a(w)(u)$ denote the length of the cycle of $S_a(w)$ containing $u$, i.e. the size of the orbit of $u$ under $S_a(w)$.

\begin{lem}\label{increasing}
Let $u',w'$ be suffixes of $u,w$ respectively with $|u'|=|w'|\geq 2$. For $a\in\{0,1\}$, we have $\Lambda_a(w)(u)\geq\Lambda_a(w')(u')$.
\end{lem}

\begin{proof}
The claim follows at once using (b) in \Cref{observe}. 
\end{proof}

\begin{thm}\label{power of 2}
Every cycle of $S_0(w)$ and $S_1(w)$ has a length which is a power of $2$. Furthermore, for $a,b,c\in \{0,1\}$, we either have $\Lambda_c(aw)(bu)=\Lambda_c(w)(u)$ or $\Lambda_c(aw)(bu)=2\Lambda_c(w)(u)$.
\end{thm}

\begin{proof}
For $|w|=2$ we have $2\leq |\calO_w|\leq 2$ by \Cref{matrix-size-bound} and \Cref{matrix-size-bound-1}. Furthermore, the only words of length $2$ which are not in $\calO_w$ are $10$ and $00$ (by \Cref{last-letter}), and both of these words are fixed by $S_0(w)$ and $S_1(w)$. Hence, every cycle of $S_0(w)$ and $S_1(w)$ has length $1$ or $2$. Now, it suffices to prove the second part of the claim for arbitrary $w$, since induction on $|w|$ shows that the second claim implies the first. By \Cref{symmetry}, we may assume, without loss of generality, that the first letter of $w$ is $0$. Let $w=0w'$, and let $d\in\{0,1\}$ be the first letter of $u$ with $u=du'$.\\

\textbf{Case I:} $c=1$.\\
By \eqref{logical-defn}, the first letter of $S_c(aw)(bu)=S_1(a0w')(bdu')$ is given by
$$(0\land d)\oplus b=b,$$
which is independent of $d$. Hence, one has
\begin{align*}
S_1(aw)(bu)&=bS_1(w)(u)\text{, therefore}\\
S_1(aw)^2(bu)&=bS_1(w)^2(u)\text{, therefore}\\
S_1(aw)^3(bu)&=bS_1(w)^3(u)\text{, therefore}\\
&\ldots
\end{align*}
This yields that $\Lambda_1(aw)(bu)=\Lambda_1(w)(u)$, as desired.\\

\textbf{Case II:} $c=0$.\\
By \eqref{logical-defn}, the first letter of $S_c(aw)(bu)=S_0(a0w')(bdu')$ is given by
$$(1\land d)\oplus b=d\oplus b.$$
Thus, one has
\begin{align}\label{match}
S_0(aw)(bu)&=(d\oplus b)S_0(w)(u).
\end{align}
Setting $\Lambda_0(w)(u)=L$, repeated applications of \eqref{match} yield that $S_0(aw)^L(bu)$ either equals $bS_0(w)^L(u)=bu$ (if $d=0$) or $\overline{b}S_0(w)^L(u)=\overline{b}u$ (if $d=1$). If the prior, then $\Lambda_0(aw)(bu)=L$ by \Cref{increasing}. If the latter, then observe that
\begin{equation*}
S_0(aw)^L(bu)=10^\ell\oplus bu.
\end{equation*}
Since $S_0(aw)(10^\ell)=10^\ell$ (by \Cref{last-letter}), one has that $\Lambda_0(aw)(bu)=2L$ (by (b) of \Cref{observe} and \Cref{increasing}). 
\end{proof}

\begin{rmk}\label{promoted-footnote}
\eqref{match} yields that $S_0(aw)^L(bu)=bu$ if the first letter of an even number of elements in the cycle of $S_0(w)$ containing $u$ is $1$, and $S_0(aw)^L(bu)=\overline{b}u$ otherwise. Note, however, that this is under the assumption that the first letter of $w$ is $0$.
\end{rmk}

\begin{lem}\label{cycle-length}
For $a\in\{0,1\}$ and $\ell\geq 2$, one has
$$\Lambda_a\left(a^\ell\right)\left(0^{\ell-1}1\right)=2^{\ceil{\log_2(\ell)}}.$$
\end{lem}

\begin{proof}
By \Cref{symmetry}, we may assume that $a=1$ without loss of generality. Let $w=1^\ell$. We will fist prove the claim for $\ell=2^r$ using induction on $r$, where the base case of $r=0$ is trivial. Suppose the claim is true for $r-1$ (for some $r\geq 1$), and we will prove it for $r$. The word
$$0^{\ell-1-2^{r-1}}1=0^{2^{r-1}-1}1$$
is a prefix of $S_1(w)^{2^{r-1}}(0^{\ell-1}1)$ (by \Cref{prefix}). Furthermore, the word $$S_1\h{-2}\left(1^{2^{r-1}}\right)^{2^{r-1}}\h{-3}\left(0^{2^{r-1}-1}1\right)$$
is a suffix of $S_1(w)^{2^{r-1}}(0^{\ell-1}1)$ (by \Cref{last-letter}). By the induction hypothesis, this means that
$$0^{2^{r-1}-1}1$$
is a suffix of $S_1(w)^{2^{r-1}}(0^{\ell-1}1)$. Combining this prefix and suffix information yields that
\begin{align}
S_1(w)^{2^{r-1}}\h{-2}\left(0^{\ell-1}1\right)&=\left(0^{2^{r-1}-1}1\right)^2\label{halfway}\\
&=0^{\ell-1}1\oplus 0^{2^{r-1}-1}10^{2^{r-1}}.\nonumber
\end{align}
Applying $S_1(w)^{2^{r-1}}$ to both sides and using (b) of \Cref{observe}, one obtains
\begin{align*}
S_1(w)^{\ell}\h{-1}\left(0^{\ell-1}1\right)&=S_1(w)^{2^{r-1}}\h{-2}\left(0^{\ell-1}1\right)\oplus S_1(w)^{2^{r-1}}\h{-2}\left(0^{2^{r-1}-1}10^{2^{r-1}}\right)\\
&=\left(0^{2^{r-1}-1}1\right)^2\oplus 0^{2^{r-1}-1}10^{2^{r-1}}\text{ (using \eqref{halfway} and the induction hypothesis)}\\
&=0^{\ell-1}1.
\end{align*}
Consequently, $\Lambda_1(w)(0^{\ell-1}1)$ is a factor of $\ell=2^r$. Also, $\Lambda_1(w)(0^{\ell-1}1)> 2^{r-1}$ by \eqref{halfway}. This proves the lemma for $\ell=2^r$. Next, suppose $2^{r-1}<\ell\leq 2^r$, for some $r>0$. On one hand, \Cref{increasing} and \Cref{power of 2}, together with what was shown above, yields that $\Lambda_1(w)(0^{\ell-1}1)$ is either $2^{r-1}$ or $2^r$. On the other hand, we know that $0^{\ell-1-2^{r-1}}1$ is a prefix of $S_1(w)^{2^{r-1}}(0^{\ell-1}1)$ (since $\ell>2^{r-1}$). In particular,
$$S_1(w)^{2^{r-1}}\h{-2}\left(0^{\ell-1}1\right)\neq 0^{\ell-1}1.$$
The lemma now follows.
\end{proof}

\begin{cor}\label{single-run-cycle-length}
Let $a\in\{0,1\}$, and suppose $u$ has last letter $1$. Then
$$\Lambda_a\h{-2}\left(a^\ell\right)\h{-2}(u)=2^{\ceil{\log_2(\ell)}}.$$
\end{cor}

\begin{proof}
Let $w=a^\ell$ and $R=\{r\in[\ell]\mid u_r=1\}$. Evidently, we have
$$u=\bigoplus_{r\in R}0^{\ell-r}10^{r-1}.$$
(b) of \Cref{observe} yields that
\begin{equation*}
S_a(w)^i(u)=\bigoplus_{r\in R}S_a(w)^i\h{-2}\left(0^{\ell-r}10^{r-1}\right).
\end{equation*}
The corollary now follows from \Cref{cycle-length}. 
\end{proof}

\begin{lem}\label{single-run}
For $a\in\{0,1\}$ and $\ell\geq 2$, the cycle of $S_a(a^\ell)$ containing $0^{\ell-1}1$ has an odd number of words with first letter $1$ if and only if $\ell$ is a power of $2$.
\end{lem}

\begin{proof}
By \Cref{symmetry}, we may assume, without loss of generality, that $a=1$. Let $w=1^\ell$ and $L=\Lambda_1(1^{\ell})(0^{\ell-1}1)$. We claim that $\Lambda_1(1w)(0^{\ell}1)=2L$ if, and only if, the cycle of $S_1(w)$ containing $0^{\ell-1}1$ has an odd number of words with first letter $1$. Using this claim, together with \Cref{cycle-length}, the statement of the lemma follows.\\

The above claim essentially follows from \Cref{promoted-footnote}, but we reproduce the argument here. Let $u\in\{0,1\}^{\ell+1},b\in\{0,1\}$, and $u'\in\{0,1\}^{\ell}$ with $u=bu'$. Let the first letter of $u'$ be $c$. From the definition \eqref{logical-defn}, the first letter of $S_1(1w)(u)$ is
\begin{align*}
(1\land c)\oplus b&=c\oplus b.
\end{align*}
Hence, one has
$$S_1(1w)(u)=(c\oplus b)S_1(w)(u').$$
Repeated applications of the above yield that if the cycle of $S_1(w)$ containing $0^{\ell-1}1$ has an even number of words with first letter $1$, then $S_1(1w)^{L}(0^{\ell}1)=0^l1$, and otherwise $S_1(1w)^{L}(0^{\ell}1)=10^{\ell-1}1$. In the first case we see that $\Lambda_1(1w)(0^\ell1)=L$, and in the second case we see that $\Lambda_1(1w)(0^\ell1)=2L$. This proves the above claim.
\end{proof}

\section{Sufficient conditions for satisfying {\em Properties P} and \emph{Q}}\label{sec6}

\begin{thm}\label{simple}
Let $a\in\{0,1\}$ and $w\in\{0,1\}^\ast$ (possibly with $|w|<2$). Let $k>1$ be a power of $2$ and $1<j\leq k$. The pair $\left(a^kw\ol{a}a^{j-1},0^{k-1}10^{|w|+j-1}1\right)$ satisfies {\em Property Q}.
\end{thm}

\begin{proof}
Let $\ell=|w|$, $w'=a^kw\ol{a}a^{j-1}$ and $u'=0^{k-1}10^{\ell+j-1}1$. We will show that $w'$ and $u'$ satisfy the hypothesis of \Cref{imp-suff-cond} by showing that the following hold.
\begin{enumerate}[(i)]
\item $S_{\ol{a}}(w')(u')=u'$.

\item $T_{\ol{a}}(w')(u')=0$.

\item The cycle of $S_a(w')$ containg $u'$ has an odd number of words with first letter $1$.
\end{enumerate}
(i) and (ii) are easy to verify (we need $j>1$ for (i) to hold and $k>1$ for (ii) to hold). For (iii), one obtains the following using (b) of \Cref{observe}.
\begin{align}
S_a(w')^i(u')&=S_a(w')^i\h{-2}\left(0^{k-1}10^{\ell+j}\right)\oplus S_a(w')^i\h{-2}\left(0^{k+\ell+j-1}1\right)\nonumber\\
&=\left(S_a\h{-1}\left(a^k\right)^i\h{-2}\left(0^{k-1}1\right)0^{\ell+j}\right)\oplus\left(0^{k+\ell}S_a\h{-1}\left(\ol{a}a^{j-1}\right)^i\h{-2}\left(0^{j-1}1\right)\right)\nonumber\\
&=\left(S_a\h{-1}\left(a^k\right)^i\h{-2}\left(0^{k-1}1\right)0^{\ell+j}\right)\oplus\left(0^{k+\ell}S_a\h{-1}\left(a^{j}\right)^i\h{-2}\left(0^{j-1}1\right)\right).\label{prefix1}
\end{align}
By \Cref{increasing}, one has
$$\Lambda_a\h{-1}\left(a^{j}\right)\h{-2}\left(0^{j-1}1\right)\leq\Lambda_a\h{-1}\left(a^{k}\right)\h{-2}\left(0^{k-1}1\right).$$
Hence, \eqref{prefix1} yields that
\begin{equation}\label{cycle1}
\Lambda_a(w')(u')=\Lambda_a\h{-1}\left(a^{k}\right)\h{-2}\left(0^{k-1}1\right).
\end{equation}
The cycle of $S_a(a^k)$ containing $0^{k-1}1$ has an odd number of words with first letter $1$ (by \Cref{single-run}, since $k$ is a power of $2$), so \eqref{prefix1} and \eqref{cycle1} show that the cycle of $S_a(w')$ containing $u'$ also has an odd number words with first letter $1$.
\end{proof}

Although \Cref{imp-suff-cond} only gives a sufficient condition for \em{Property P} to to be satisfied, below we demonstrate how the linear algebraic machinery developed so far can also be used to examine its failure in some simple situations.

\begin{thm}\label{one-run}
Let $a\in\{0,1\}$ and $\ell\geq2$. \em{Property P} is satisfied by $w:=a^\ell$ if and only if $\ell$ is a power of $2$.
\end{thm}

\begin{proof}
If $\ell$ is a power of $2$, the claim follows using \Cref{single-run} and \Cref{imp-suff-cond}. Next, suppose $\ell$ is not a power of $2$. Since $S_{\overline{a}}(w)$ is the identity permutation, one has
$$\calO_w=\left\{S_a(w)^i\h{-2}\left(0^{\ell-1}1\right)\mid i\geq 0\right\}.$$
Using this observation together with \Cref{cycle-length}, we may index the elements of $\calO_w$ as
$$u_i:=S_a(w)^{i}\h{-2}\left(0^{\ell-1}1\right)=\sigma_{M_a(w)}^i\h{-2}\left(0^{\ell-1}1\right)\text{, for }L=\ceil{\log_2(\ell)}\text{ and }0\leq i<2^{L}.$$
From \Cref{single-run} and \Cref{sign-cyclic-eigenval}, we know that $M_a(w)$ has an eigenvector $x$ with eigenvalue $1$ and entries in $\{1,-1\}$. Since $\sigma_{M_{a}(w)}=(u_0\,u_2\,\ldots\, u_{2^L-1})$ (in cyclic notation), we can use the same idea as that used in the proof of \Cref{sign-cyclic-eigenval} to find this eigenvector explicitly. Let $R$ be the set of $r\in[2^L]$ such that the $u_r$-th row of $M_a(w)$ has a $-1$. Set the $u_q$-th entry $x(u_q)$ of $x$ to
$$x(u_q)=(-1)^{\hash\{r\geq q\,\mid\, r\in R\}}.$$
The $u_r$-th row of $M_a(w)$ has a $-1$ if and only if $T_a(u_{r})=1$, so one obtains that
$$x(u_q)=(-1)^{\hash\{r\geq q\,\mid\,T_a(u_{r})=1\}}.$$
Also, $T_a(u_r)=1$ if and only if $u_r$ has first letter $1$, so
$$x(u_q)=(-1)^{\hash\{r\geq q\,\mid\,u_r\text{ has first letter $1$}\}}.$$
$0^{\ell-1-r}$ is a prefix of $S_a(w)^{r}(0^{\ell-1}1)=u_r$ for $0\leq r<\ell$ (by \Cref{prefix}), so $1$ cannot be the first letter of $u_r$ for $0\leq r<\ell-1$. Thus, one has
\begin{align}\label{obs1}
x(u_q)=(-1)^{|R|}\text{ for }0\leq q\leq \ell-1.
\end{align}
Using \Cref{matrix-recurrence}, together with induction on $n$, yields the following, where $M=M(w), V=V(w),=v(w)$ and $c=c(w)$.
\begin{align}
V\hspace{-2pt}\left(2^n-1\right)&=M^nV(0)+\sum_{i=0}^{n-1}M^ic\nonumber\\
&=M^nv(0)+\sum_{i=0}^{n-1}M^i((I-M)v(0)+v(1))\nonumber\\
&=v(0)+\sum_{i=0}^{n-1}M^{i}v(1)\label{recurrence}.
\end{align}
We will now show that $\langle v(w)(1),x\rangle\neq 0$, where $\langle\cdot,\cdot\rangle$ is the standard inner product. Since $M(w)=M_a(w)+M_{\overline{a}}(w)=M_a(w)+I$, and $M_a(w)$ is an orthonormal matrix with all eigenvalues distinct (by \Cref{sign-cyclic-eigenval}), $M(w)$ has an orthomormal eigenbasis and all eigenvalues distinct. Furthermore, $2$ is the only possible eigenvalue of $M(w)$ with absolute value at least $2$ (by (c) of \Cref{eigenval-2}). Consequently, if we show that $\langle v(w)(1),x\rangle\neq 0$, then, by \eqref{recurrence}, there would exist $\delta>0$ such that all entries of $V(w)\hspace{-2pt}\left(2^n-1\right)$ have absolute value at least $2^n\delta$ (since all entries of $x$ are non-zero). This would prove the theorem.\\

To show that $\langle v(w)(1),x\rangle\neq 0$, it suffices to show that at least $2^{L-1}+1$ of the $2^L$ terms in the component-wise sum for $\langle v(w)(1),x\rangle$ are equal (since all terms take values $\pm1$).\\

\textbf{Case I:} $a=0$.\\
All entries of $v(w)(1)$ are $1$. Combining this with \eqref{obs1}, one has that at least $\ell>2^{L-1}$ terms in the component-wise sum for $\langle v(w)(1),x\rangle$ are equal.\\

\textbf{Case II:} $a=1$.\\
The $u_q$-th entry of $v(w)(1)$ is $1$ if the first letter of $u_q$ is $0$ and it is $-1$ otherwise. Since $0^{\ell-1-r}$ is a prefix of $u_r$ for $0\leq r<\ell-1$ (as noted previously), the $u_q$-th entry of $v(w)(1)$ is $1$ for $0\leq q<\ell-1$.\\

\textbf{Case IIa:} $\ell\geq 2^{L-1}+2$.\\
Combining the above observation with \eqref{obs1}, we see that at least $\ell-1>2^{L-1}$ terms in the component-wise sum for $\langle v(w)(1),x\rangle$ are equal.\\

\textbf{Case IIb:} $\ell=2^{L-1}+1$.\\
By the argument used to prove \Cref{matrix-size-bound-1}, we know that
$$S_1(w)^{2^{L-1}}\hspace{-3pt}\left(0^{2^{L-1}}1\right)$$
has first letter $1$. By \Cref{last-letter}, it also has suffix
$$S_1\hspace{-2pt}\left(1^{2^{L-1}}\right)^{2^{L-1}}\hspace{-3pt}\left(0^{2^{L-1}-1}1\right).$$
By \Cref{cycle-length}, this suffix is just $0^{2^{L-1}-1}1$. Hence, one has 
$$S_1(w)^{2^{L-1}}\hspace{-3pt}\left(0^{2^{L-1}}1\right)=10^{2^{L-1}-1}1.$$
By (b) of \Cref{observe}, this yields
\begin{align*}
S_1(w)^{2^{L-1}+r}\hspace{-2pt}\left(0^{2^{L-1}}1\right)&=S_1(w)^{r}\hspace{-2pt}\left(10^{2^{L-1}}\right)\oplus S_1(w)^r\hspace{-2pt}\left(0^{2^{L-1}}1\right)\\
&=10^{2^{L-1}}\oplus S_1(w)^r\hspace{-2pt}\left(0^{2^{L-1}}1\right).
\end{align*}
The second term in the above has prefix $0^{\ell-1-r}$ for $0\leq r<\ell-1=2^{L-1}$ (as noted previously), so, in toto, the first letter of
$$S_1(w)^{r}\hspace{-2pt}\left(0^{2^{L-1}}1\right)$$
is $1$ for $2^{L-1}\leq r<2^L$. Combining this with the earlier observation for case II, one has
\begin{align*}
u_q\text{-th entry of }v(w)(1)=\begin{cases}
1 & 0\leq q< 2^{L-1}\\
-1 & 2^{L-1}\leq q<2^L.
\end{cases}
\end{align*}
Not all entries of $x$ are equal (since $T_1(w)(u)=1$ for at least one $u\in\calO_w$), so, by \eqref{obs1}, there exists an integer $q_0$, with $2^{L-1}<q_0<2^L$, such that the $u_{q_0}$-th entry of $x$ is $-(-1)^{|R|}$. Hence, one has that at least $2^{L-1}+1$ terms in the component-wise sum for $\langle v(w)(1),x\rangle$ are equal.
\end{proof}

\begin{thm}\label{long-prefix}
Let $w\in\{0,1\}^{\ast}$ (possibly with $|w|<2$), $a\in\{0,1\}$, and $k$ be a power of $2$ such that $a^k$ is not a factor of $w$. {\em Property Q} is satisfied by $\left(a^k\overline{a}w,0^{k+|w|}1\right)$. In particular, $w$ satisfies {\em Property P}.\\
More generally, if $u\in\{0,1\}^{|w|}, u\neq 0^{|w|})$, and $b\in\{0,1\}$ such that $S_b(w)(u)=u$ and $T_b(w)(u)=0$, then $\left(a^k\overline{a}w,0^{k+1}u\right)$ satisfies {\em Property Q}.
\end{thm}

\begin{rmk}
The condition $S_b(w)(u)=u$ is equivalent to saying that $w_r=\ol{b}$ whenever $1\leq r\leq |u|$ with $u_r=1$. The condition $T_b(w)(u)=0$ is equivalent to saying that $b$ is not the first letter of $w$ or $1$ is not the first letter of $u$ (cf. the discussion just before \Cref{imp-suff-cond}).
\end{rmk}

{\em In order to apply \Cref{imp-suff-cond} to prove \Cref{long-prefix}, we require some lemmas.}

\begin{lem}\label{opposite}
Let $a\in\{0,1\}$, and $w=\overline{a}^{j_m}a^{i_m}\ldots \overline{a}^{j_1}a^{i_1}$ with $i_r,j_r\geq1$ for all $r$. All cycles of $S_a(w)$ have length at most $\max\limits_{r}\,2^{\ceil{\log_2(i_r+1)}}$. All cycles of $S_{\overline{a}}(w)$ have length at most $ \max\left\{2^{\ceil{\log_2(j_m)}},2^{\ceil{\log_2(j_r+1)}}\mid r<m\right\}$.
\end{lem}

\begin{proof}
Without loss of generality we may assume that $a=1$, by \Cref{symmetry}. Hence, one has
$$w=0^{j_m}1^{i_m}\ldots 0^{j_1}1^{i_1}.$$
Let $u\in\{0,1\}^{|w|}$ and write $u=v_mu_m\ldots v_1u_1$ for words $u_r\in\{0,1\}^{i_r}$ and $v_r\in\{0,1\}^{j_r}$. We wish to prove the following.
\begin{align}
\Lambda_{1}(w)(u)&\leq\max\limits_{r}\,2^{\ceil{\log_2(i_r+1)}}\label{1-eq}\\
\Lambda_{0}(w)(u)&\leq \max\left\{2^{\ceil{\log_2(j_m)}},2^{\ceil{\log_2(j_r+1)}}\mid r<m\right\}.\label{0-eq}
\end{align}
We will only prove \eqref{1-eq}, since the proof of \eqref{0-eq} is essentially identical. The main idea is to apply (b) of \Cref{observe}, for which we observe the following. Let $b\in\{0,1\}$ and $x,y,z\in\{0,1\}^\ast$ with $|y|\geq2$. By \eqref{logical-defn} we have
\begin{align*}
S_1\hspace{-2pt}\left(x01^{|y|}z\right)\h{-2}\left(0^{|x|}by0^{|z|}\right)&=0^{|x|}\,S_1\h{-2}\left(01^{|y|}\right)\h{-2}(by)\h{1}0^{|z|}.
\end{align*}
Furthermore, \eqref{logical-defn} yields $S_1(01^{|y|})=S_1(1^{|y|+1})$. Hence the above can be stated as
\begin{align*}
S_1\h{-2}\left(x01^{|y|}z\right)\h{-2}\left(0^{|x|}by0^{|z|}\right)&=0^{|x|}\,S_1\h{-2}\left(1^{|y|+1}\right)\h{-2}(by)\,0^{|z|}.
\end{align*}
By repeated applications of the above, we obtain
\begin{align}\label{block-action}
S_1\h{-2}\left(x01^{|y|}z\right)^i\h{-2}\left(0^{|x|}by0^{|z|}\right)&=0^{|x|}\,S_1\h{-2}\left(1^{|y|+1}\right)^i(by)\,0^{|z|}.
\end{align}
Now we apply (b) of \Cref{observe} as follows to break the problem at hand into simpler pieces which can be analysed using the above.
\begin{align*}
S_1(w)^i(u)&=S_1(w)^i\h{-2}\left(v_m0^{i_m}\ldots v_20^{i_2}v_10^{i_1}\right)\h{-1}\oplus\bigoplus_{r=1}^mS_1(w)^i(0\ldots0u_r0\ldots0)\\
&=v_m0^{i_m}\ldots v_20^{i_2}v_10^{i_1}\oplus\bigoplus_{r=1}^mS_1(w)^i(0\ldots0u_r0\ldots0).
\end{align*}
The number of $0$'s to the left of $u_r$ in the above is $(j_m+\ldots+j_r)+(i_m+\ldots+i_{r+1})$. Likewise, the number of $0$'s to the right of $u_r$ is $(j_{r-1}+\ldots j_1)+(i_{r-1}+\ldots i_1)$. Hence \eqref{block-action} yields
$$S_1(w)^i(u)=v_m0^{i_m}\ldots v_20^{i_2}v_10^{i_1}\oplus\bigoplus_{r=1}^m0\ldots0\,S_1\h{-2}\left(1^{|u_r|+1}\right)^i(0u_r)\,0\ldots0.$$
Here the number of $0$'s to the left is $(j_m+\ldots+j_r)+(i_m+\ldots+i_{r+1})-1$\footnote{That this quantity is non-negative for all $r$ is important to note, since it is the reason behind the slight difference between the form of the right sides of \eqref{1-eq} and \eqref{0-eq}.} and to the right is $(j_{r-1}+\ldots j_1)+(i_{r-1}+\ldots i_1)$. Since $|u_r|=i_r$, \eqref{1-eq} now follows using \Cref{single-run-cycle-length}. 
\end{proof}

\begin{lem}\label{same}
Let $a\in\{0,1\}$ and $w=a^{i_{m+1}}\overline{a}^{j_m}a^{i_m}\ldots \overline{a}^{j_1}a^{i_1}$ with $i_r,j_r\geq1$. All cycles of $S_a(w)$ have length at most $ \max\left\{2^{\ceil{\log_2(i_{m+1})}},2^{\ceil{\log_2(i_r+1)}}\mid r\leq m\right\}$. All cycles of $S_{\overline{a}}(w)$ have length at most $\max\limits_{r}\,2^{\ceil{\log_2(j_r+1)}}$.
\end{lem}

\begin{proof}
Identical to the proof of \Cref{opposite}.
\end{proof}

\begin{cor}\label{upper-bound}
If $a\in\{0,1\}$ and $k$ is a power of $2$ such that $a^k$ is not a factor of $w$, then all cycles of $S_a(w)$ have length at most $k$.
\end{cor}

\begin{proof}
If $w$ starts and ends with different letters, apply \Cref{opposite}. Otherwise, apply \Cref{same}.
\end{proof}

\begin{proof}[Proof of \Cref{long-prefix}]
Let $\ell=|w|=|u|$. It is clear that the first statement follows from the second by \Cref{imp-remark}, so it suffices to prove the second statement. We will show that $w':=a^k\overline{a}w$ and $u':=0^{k+1}u$ satisfy the hypothesis of \Cref{imp-suff-cond} by showing that the following hold.
\begin{enumerate}[(i)]
\item $S_b(w')(u')=u'$.

\item $T_b(w')(u')=0$.

\item There exists a word $u''\in\calO_{w'}(u')$ such that the cycle of $S_a(w')$ containg $u''$ has an odd number of words with first letter $1$.
\end{enumerate}

By hypothesis on $u$, it is easy to see that (i) holds. Since $u'$ has first letter $0$, we see that (ii) holds. To construct $u''$ as in (iii), suppose $u$ has prefix $0^{r-1}1$ for some $r\geq 1$ (since $u\neq0^\ell$). Let $h=\overline{a}w_\ell\ldots w_{\ell-r}$ and define
$$u''=S_{h}(w')(u')\in\calO_{w'}(u')$$
We see that $u''$ has prefix $0^{k-1}1$, so in particular we have
$$u''=0^{k-1}1\,S_{h}(\overline{a}w)(0u)$$
Using the fact that $\ol{a}$ is a prefix of $\ol{a}w$, this yields
\begin{align}\label{split1}
S_a(w')^i(u'')=\left[S_a\h{-2}\left(a^k\right)^i\h{-2}\left(0^{k-1}1\right)\right]\left[S_a(\ol{a}w)^i\circ S_{h}(\ol{a}w)(0u)\right]
\end{align}
By \Cref{cycle-length} we have $\Lambda_a(a^k)(0^{k-1}1)=k$ (since $k$ is a power of $2$). Since $a^k$ is not a factor of $\ol{a}w$, \Cref{upper-bound} yields $\Lambda_a(\ol{a}w)[S_{h}(\ol{a}w)(0u)]\leq k$. Combining these two facts using \eqref{split1}, we obtain $\Lambda_a(w')(u'')=k$ — in particular
\begin{equation}\label{same-cycle-length}
\Lambda_a(w')(u'')=\Lambda_a\h{-2}\left(a^k\right)\h{-3}\left(0^{k-1}1\right)
\end{equation}
Since the cycle of $S_a(a^k)$ containing $0^{k-1}1$ has an odd number of words with first letter $1$ (by \Cref{single-run}, since $k$ is a power of $2$), \eqref{split1} and \eqref{same-cycle-length} show that the cycle of $S_a(w')$ containing $u''$ also has an odd number of words with first letter $1$.
\end{proof}

\begin{thm}\label{two-runs}
Let $a\in\{0,1\}$, $j,k\geq1$ and $w=a^j\overline{a}^k$. There exists $\epsilon>0$ such that
$$\sum_{n=0}^N(-1)^{s_w(n)}\in O\left(N^{1-\epsilon}\right)$$
More generally, for $u\in\{0,1\}^k$ $(u\neq0^k)$ there exists $\epsilon>0$ such that
$$\left\|V\h{-2}\left(w,0^{j}u\right)\h{-2}(N)\right\|\in O\left(N^{1-\epsilon}\right)$$
\end{thm}

The proof requires some lemmas.

\begin{lem}\label{direct-sum}
Let $j\geq1$ with binary expansion
$$j=2^{p_1}+\ldots+2^{p_r},\,0\leq p_1<\ldots<p_r$$
To each set $I\subset [r]$, associate a number $j_I$ given by
$$j_I=\sum_{i\in I}2^{p_i}$$
For $a\in\{0,1\}$ we have
$$\bigoplus_{I\subset [r]}S_a\h{-2}\left(a^{j+1}\right)^{j_I}\h{-3}\left(0^{j}1\right)=10^j$$
\end{lem}

\begin{proof}
Let $w=a^{j+1}$ and $u^I=S_a(w)^{j_I}(0^k1)$ with letters
$$u^I=u^I_{j}\ldots u^I_0$$
We claim that
\begin{equation}\label{claim-I}
u^I_i=\begin{cases}
1 & i=j_J\text{ for some }J\subset I\\
0 & \text{otherwise}
\end{cases}
\end{equation}
Using the above, the lemma can be proved by a straighforward counting argument as follows. Let
\begin{align*}
u=u_j\ldots u_0&=\bigoplus_{I\subset [r]}u^I\\
\implies u_i&=\bigoplus_{I\subset [r]}u^I_i
\end{align*}
If $i$ is not of the form $j_J$ for $J\subset[r]$, then $1$ does not appear in the above sum and so $u_i=0$. If $i=j_J$ for some $J\subset[r]$, then the number of times $1$ appears in the sum equals the number of sets $J'\subset [r]$ which contain $J$. This equals $2^{r-|J|}$ which is odd only when $|J|=r$, i.e. $J=[r]$ and $i=j_{[r]}=j$. Hence $u_i=1$ if and only if $i=j$, i.e. $u=10^j$.\\ 

It now remains to prove \eqref{claim-I}, for which we will use induction on $r$. For the base case $r=1$ we see that $j=2^{p_1}$ is a power of $2$ and $j_\emptyset=0$, $j_{\{1\}}=j$. \eqref{claim-I} is trivial for $I=\emptyset$, so it suffices to verify it for $I=\{1\}$. The word
$$u^{\{1\}}=S_a(w)^j\h{-1}\left(0^k1\right)$$
has first letter $1$ and suffix
$$S_a\h{-2}\left(a^j\right)^j\h{-2}\left(0^{j-1}1\right)$$
By \Cref{cycle-length} this suffix is just $0^{j-1}1$, so we have
\begin{equation}\label{special-case}
u^{\{1\}}=S_a(w)^j\h{-2}\left(0^k1\right)=10^{j-1}1
\end{equation}
Now suppose $r>1$ and that the claim is true for $j'=2^{p_1}+\ldots+2^{p_{r-1}}$. We will prove the claim for $j=j'+2^{p_{r}}$. If $I\subset[r-1]$, then $j_I\leq j'$ and so we see that
\begin{equation}\label{without-largest}
S_a(w)^{j_I}\h{-2}\left(0^{j}1\right)=0^{2^{p_r}}S_a(w)^{j_I}\h{-2}\left(0^{j'}1\right)
\end{equation}
Hence \eqref{claim-I} follows by the induction hypothesis. Next suppose $I=\{r\}\cup I'$ for some $I'\subset [r-1]$, so $j_I=2^{p_r}+j_{I'}$. We have
\begin{align*}
S_a(w)^{2^{p_r}}\h{-2}\left(0^k1\right)&=0^{j'}S_a(w)^{2^{p_r}}\h{-2}\left(0^{2^{p_r}}1\right)\\
&=0^{j'}10^{2^{p_r}-1}1
\end{align*}
Here, the second equality essentially comes from \eqref{special-case}. Now applying $S_a(w)^{j_{I'}}$ to both sides and using the induction hypothesis proves \eqref{claim-I}.
\end{proof}

\begin{lem}\label{general}
Let $a\in\{0,1\}$, $j,k\geq 1$, $w=a^j\ol{a}^k$ and $u\in\{0,1\}^k$ $(u\neq 0^k)$. There exists a word in $\calO_w(0^ju)$ with prefix $10^{j-1}$.
\end{lem}

\begin{proof}
Let $q_0=0$ and $q_1<q_2<\ldots$ be all the values of $q>0$ such that
$$T_{\ol{a}}\h{-1}\left(\ol{a}^j\right)\circ S_{\ol{a}}\h{-1}\left(\ol{a}^j\right)^{q}\h{-1}(u)=1$$
It can be seen that the sequence $(q_i)_{i\geq1}$ is non-empty, and hence it is also an infinite periodic sequence (since $S_{\ol{a}}(\ol{a}^k)$ is a permutation of $\{0,1\}^k$ of finite order). Using \eqref{logical-defn}, we see that
$$S_{\ol{a}}(w)^{q_i-q_{i-1}}\h{-2}\left[0^{j}S_{\ol{a}}\h{-1}\left(\ol{a}^k\right)^{q_{i-1}}\h{-1}(u)\right]$$
has prefix $0^{j-1}1$. Hence, induction on $i\geq1$ using \Cref{observe} yields the following for $v\in\{0,1\}^j$.
\begin{equation}\label{difference}
S_{\ol{a}}(w)^{q_i-q_{i-1}}\left[v\,S_{\ol{a}}\h{-1}\left(\ol{a}^k\right)^{q_{i-1}}\h{-1}(u)\right]=\left(v\oplus0^{j-1}1\right)S_{\ol{a}}\h{-1}\left(\ol{a}^k\right)^{q_i}\h{-1}(u)
\end{equation}
Also, for $p\geq0$ and $x\in\{0,1\}^k$ we have
\begin{equation}\label{padding}
S_a(w)^p(vx)=S_a\h{-1}\left(a^j\right)^p\h{-2}(v)\,x
\end{equation}
Using \eqref{difference}, \eqref{padding} and (b) of \Cref{observe} yields the following for arbitrary non-negative integers $p_1,\ldots,p_n$ by induction on $n$, where $\pi=S_a(w)$ and $\tau=S_{\ol{a}}(w)$.
\begin{align}
&\quad\pi^{p_1}\circ\tau^{q_{n}-q_{n-1}}\circ\ldots\circ\pi^{p_n}\circ\tau^{q_1-q_0}\left(0^{j}u\right)\nonumber\\
&=\left(\bigoplus_{i=1}^{n}S_a\h{-2}\left(a^j\right)^{p_1+\ldots+p_i}\h{-2}\left(\ol{a}^{j-1}1\right)\right)\h{-2}S_{\ol{a}}\h{-2}\left(\ol{a}^k\right)\h{-2}(u)\label{final}
\end{align}
With notation as in \Cref{direct-sum}, we now arrange the elements of the set $\{j_I\mid I\subset[r]\}$ in ascending order (note that $j_I\neq j_J$ for $I\neq J$) as $P_0<P_1<\ldots<P_n$. Set $p_i=P_i-P_{i-1}$ for $i\geq 1$, so that the corresponding word obtained using \eqref{final} has prefix $10^{j-1}$ by \Cref{direct-sum}. This word also clearly lies in $\calO_w(0^ju)$.
\end{proof}

\begin{proof}[Proof of \Cref{two-runs}]
It is clear that the first statement follows from the second by setting $u=0^{k-1}1$, so it suffices to prove the second statement. For this we will show that $w$ and $0^ju$ satisfy the hypothesis of \Cref{imp-suff-cond} by showing that the following hold.

\begin{enumerate}[(i)]
\item $S_a(w)(0^ju)=0^ju$.
\item $T_a(w)(0^ju)=0$.
\item There exists a word $u'\in\calO_w(0^ju)$ such that $S_{\ol{a}}(w)(u')=u'$ and the first letter of $u'$ is $1$.
\end{enumerate}

(i) and (ii) follow immediately. For (iii) let $u'\in\calO_w(0^ju)$ with prefix $10^{j-1}$, whose existence is guaranteed by \Cref{general}.
\end{proof}

\section{Conclusion}\label{sec7}

In this article, we find several classes of subword-counting sequences satisfying the {\em Properties $P$, $Q$}. However, the authors cannot give a necessary and sufficient condition for these two properties. The major technique used in this article is to argue that the associated matrix has all eigenvalues strictly smaller than $2$ in absolute value, however, one may expect that there are some subword-counting sequences satisfying {\em Properties $P$} but having $2$ as eigenvalue for their matrices. Thus, in the further work, one may want to find a necessary and sufficient condition for the {\em Properties $P$, $Q$}. Also, one may want to find more classes of subword-counting sequences with the {\em Properties $P$, $Q$}.


\bibliographystyle{splncs03}
\bibliography{biblio}

\end{document}